\documentclass{amsart}
\usepackage[utf8]{inputenc}
\usepackage[T1]{fontenc}

\usepackage[final]{graphicx}
\usepackage{amsmath,amsthm,amssymb}
\usepackage{enumerate}
\usepackage{vaucanson-g}
\usepackage{cleveref}

\def\N{\mathbb N}
\def\A{\mathcal A}

\def\C{\mathcal C}

\def\u{\mathbf u}
\def\vv{\mathbf v} %\v nemuzu, to je hacek...

\renewcommand{\L}[1][]{ \mathcal{L}_{#1} (\mathbf{u}) }

\def \Rk#1 {$\mathcal{R}_{#1}$}

\def \FC#1 {
\mathcal{C}
\ifthenelse{\equal{#1}{}}{}{(#1)}
}

\def \PC#1 {
\mathcal{P}_{\Theta}
\ifthenelse{\equal{#1}{}}{}{(#1)}
}

\def \PCn#1 {
\mathcal{P}
\ifthenelse{\equal{#1}{}}{}{(#1)}
}

\newtheorem{thm}{Theorem}

\newtheorem{coro}[thm]{Corollary}

\newtheorem{lem}[thm]{Lemma}

\newtheorem{prop}[thm]{Proposition}
\newtheorem{defi}[thm]{Definition}

\crefname{thm}{theorem}{theorems}
\crefname{thrm}{theorem}{theorems}
\crefname{coro}{corollary}{corollaries}
\crefname{example}{example}{examples}
\crefname{lem}{lemma}{lemmas}
\crefname{lmm}{lemma}{lemmas}
\crefname{claim}{claim}{claims}
\crefname{obs}{obsertvation}{observations}
\crefname{proposition}{proposition}{propositions}
\crefname{prop}{proposition}{propositions}
\crefname{defi}{definition}{definitions}

\theoremstyle{remark}

\newtheorem{priklad}[thm]{Example}

\crefname{priklad}{example}{examples}

\def\pf{\begin{proof}}
\def\pfk{\end{proof}}

\begin{document}
%%%%%%%%%%%%%%%%%%%%%%%%%%%%%%%%%%%%%%%%%%%%%%%%%%%%%%%%%%%%%%%%%%%%%%%%%%%%%%%%%%%%%%%
% Zacatek dokumentu %%%%%%%%%%%%%%%%%%%%%%%%%%%%%%%%%%%%%%%%%%%%%%%%%%%%%%%%%%%%%%%%%%%
%%%%%%%%%%%%%%%%%%%%%%%%%%%%%%%%%%%%%%%%%%%%%%%%%%%%%%%%%%%%%%%%%%%%%%%%%%%%%%%%%%%%%%%
\title{On Theta-palindromic Richness}
\author{Štěpán Starosta}
\address{Department of Mathematics, FNSPE, Czech Technical University in Prague, Trojanova 13, 120~00 Praha~2, Czech Republic}

\date{\today}
\email{staroste@fjfi.cvut.cz}

\begin{abstract}
In this paper we study generalization of the reversal mapping realized by an arbitrary involutory antimorphism $\Theta$.
It generalizes the notion of a palindrome into a $\Theta$-palindrome -- a word invariant under $\Theta$.
For languages closed under $\Theta$ we give the relation between $\Theta$-palindromic complexity and factor complexity.
We generalize the notion of richness to $\Theta$-richness
and we prove analogous characterizations of words that are $\Theta$-rich, especially in the case of set of factors invariant under $\Theta$.
A criterion for $\Theta$-richness of $\Theta$-episturmian words is given together with other examples of $\Theta$-rich words.
\end{abstract}

\maketitle

\section{Introduction}

A palindrome is a word that reads the same from the left as from the right.
In the last decade the study of palindromes  in the field of combinatorics on words has notably grown.
The motivation comes for instance from the application in mathematical physics or diophantine approximation.
Two notions were mainly studied.
First the {\bf palindromic complexity} which counts the number of palindromes of given length in an infinite word.
Its study resulted for example in a new equivalent definition of Sturmian words using palindromic complexity (\cite{DrPi}).
The second notion is the notion of {\bf defect} which measures the saturation by palindromic factors of a finite word (\cite{BrHaNiRe}) .
An important class of words that are ``rich'' in palindromes in the utmost sense (called \textbf{rich} words) have been defined in \cite{DrJuPi,GlJuWiZa}.
They have been further studied and many characterizations have been found.
For recent results and summary on rich words see for instance \cite{BuLuGlZa,BuLuGlZa2,GlJuWiZa}.

A more general approach to the notion of palindrome has been considered by replacing the reversal mapping by a mapping $\Theta$ having the needed properties, i.e., being an involution and antimorphism.
A word which is invariant under such mapping is called a {\bf pseudopalindrome} or {\bf $\Theta$-palindrome}.
See for instance \cite{AnZaZo} or \cite{BuLuLuZa2}.
In this work we concentrate on $\Theta$-palidromic saturation in infinite words.
We derive an upper bound for the number of distinct $\Theta$-palindromic factors in a finite word.
We define an analogue to rich words -- words containing the maximum number of $\Theta$-palindromes possible.
For this class of words that are rich in $\Theta$-palindromes (or shortly {\bf $\Theta$-rich}) we prove similar equivalent statements as are known for rich words.
The main result concerns words with set of factors closed under $\Theta$.
We prove the following relation between $\Theta$-palindromic complexity $\PC{n}$ and factor complexity $\FC{n}$ which is analogous to the case of the reversal mapping (\cite{BaMaPe}).
$\PC{n}$ denotes the function counting the number of distinct $\Theta$-palindromic factors of length $n$ in a given word
and $\FC{n}$ denotes the function counting the number of distinct factors of length $n$ in a given word.

\begin{thm}
\label{nerovnost}
Let $\u$ be an infinite word with language closed under $\Theta$.
Then
$$
\PC{n} + \PC{n+1} \leq \FC{n+1} - \FC{n} + 2, \, \, \text{for all $n \in \N$.}
$$
\end{thm}

We show that the equality is attained for $\Theta$-rich words,
 a result analogous to the case of the reversal mapping in \cite{BuLuGlZa}.

\begin{thm} \label{pseudo_rovnost_rich}
Let $\u$ be an infinite word with language closed under $\Theta$.
Then $\u$ is $\Theta$-rich if and only if
$\PC{n} + \PC{n+1} = \FC{n+1} - \FC{n} + 2$ for all $n \geq 1$.
\end{thm}

We also make a connection to $\Theta$-episturmian words that are a generalization of episturmian words -- a well known example of rich words.
For a survey on such class of words see for instance \cite{BuLuLuZa}.
Unlike episturmian words, $\Theta$-episturmian words are not all $\Theta$-rich.
We give a condition when this is true.

Section \ref{sec:note} summarizes notation and introduces some needed results from combinatorics on words.
It includes known results on rich words and defines $\Theta$-palindromes.
Section \ref{sec:pcfc} establishes a relation between $\Theta$-palindromic and factor complexity.
Next section introduces the definition of $\Theta$-richness and a proof of \Cref{pseudo_rovnost_rich}.
Section \ref{sec:rich_examples} provides examples of $\Theta$-rich words and contains also examples of $\Theta$-episturmian words which are not $\Theta$-rich.

\section{Notation and basic facts}
\label{sec:note}

By $\A$ we denote an \textit{alphabet} -- a finite set of symbols, called \textit{letters}.
A finite sequence of letters $w = w_0w_1w_2 \ldots w_{n-1}$ is called a \textit{word}.
Its length is denoted by $|w|$ and equals $n$.
The set of all words over an alphabet $\A$, together with the \textit{empty word} $\varepsilon$ and the operation of \textit{concatenation}, form a free monoid $\A^{*}$.

The \textit{reversal} of a finite word $w = w_0w_1 \ldots w_{n-1}$ is
the word $\overline{w} = w_{n-1} w_{n-2} \ldots w_0$.
The mapping $w \mapsto \overline{w}$ is called reversal
and clearly is a bijection on $\A^*$.
If $w = \overline{w}$ then such word $w$ is called a \textit{palindrome}.

An infinite sequence $\u = u_0u_1u_2 \ldots$ of letters from $\A$ is said to be an \textit{infinite word} over $\A$.
A finite word $v$ is called a \textit{factor} of a finite word $w$ if there exist words $v'$ and $v''$ such that
$w = v'vv''$. If $v' = \varepsilon$ then $v$ is called a \textit{prefix} of $w$. If $v'' = \varepsilon$ then $v$ is called a \textit{suffix}.
A factor of an infinite word $\u$ is defined analogously with $v''$ being infinite.

The \textit{language} $\L$ of an infinite word $\u$ is the set of all its factors.
$\L[n]$ denotes the set of factors of $\u$ of length $n$.
We have $\L[0] = \{ \varepsilon \}$ and $\L = \bigcup_{n \in \N} \L[n]$.
If for every factor $w \in \L$ the reversal image of $w$ is in $\L$, then we say that the language $\L$ is \textit{closed under reversal}.

For any factor $w \in \L$ there exists an index $i$ such that $w = u_iu_{i+1} \ldots u_{i+|w|-1}$.
Such an index is called an \textit{occurrence} of $w$ in $\u$.
We will say that a factor $v$ is \textit{unioccurent} in some word $w$ if $v$ occurs exactly once in $w$.
If every factor of $\u$ has at least $2$ occurrences in $\u$ then the word $\u$ is said to be \textit{recurrent}.
One can verify that this is equivalent with every factor occurring infinitely many times.
For a recurrent infinite word we define a \textit{complete return word}\footnote{A shorter term ``complete return'' is also used.} to a factor $w \in \L$
to be a factor with exactly $2$ distinct occurrences of $w$, one as a prefix and one as a suffix.

A factor $w \in \L$ is said to be \textit{left (right) special} if there exist two distinct letters $a,b \in \A$
such that $aw \in \L$ and $bw \in \L$ ($wa \in \L$ and $wb \in \L$).
A factor is \textit{bispecial} if it is right and left special.

The \textit{factor complexity} of an infinite word $\u$ is a mapping $\N \mapsto \N$ defined as $\FC{n} := \# \L[n]$.
The \textit{palindromic complexity} is a mapping $\N \mapsto \N$
 defined as $\PCn{n} := \# \left \{ w \in \L[n] \mid w = \overline{w} \right \}$.

In \cite{BaMaPe} the following connection between palindromic complexity
 and the first difference of factor complexity $\Delta \FC{n} = \FC{n+1} - \FC{n}$
 was proved.

\begin{thm}[\cite{BaMaPe}] \label{nerovnost-palindromy}
Let $\u$ be an infinite word with language closed under reversal.
Then for all $n$ we have
\begin{equation}\label{eq:norm_ineq}
\PCn{n} + \PCn{n+1} \leq \Delta \FC{n} + 2.
\end{equation}
\end{thm}

Given an infinite word $\u$ and $n \in \N$, we define the \textit{Rauzy graph of order $n$ of the word $\u$}, denoted $\Gamma_n(\u)$, as follows.
$\Gamma_n(\u)$ is a directed graph whose set of vertices equals $\L[n]$ and set of edges equals $\L[n+1]$.
There is an edge $e \in \L[n+1]$ going from the vertex $v_1$ to the vertex $v_2$ if $v_1$ is a prefix of $e$ and $v_2$ is a suffix of $e$.
If the word $\u$ is recurrent then for all $n$ the graph $\Gamma_n(\u)$ is strongly connected, i.e., there exists a directed path from any vertex to any other vertex.
Let ${\rm deg}_{+}(w)$ denote the outdegree of a vertex $w$.
Especially for $\Gamma_n(\u)$ we have ${\rm deg}_{+}(w) = \# \left \{ a \in \A \mid wa \in \L \right \}$.
Rauzy graphs are a useful tool to calculate the complexity as we have
$$
\Delta \FC{n} = \sum_{w \in \L[n]} \left ( {\rm deg}_{+}(w) - 1 \right ).
$$

An \textit{$n$-simple path} $p$ is a factor of $\u$ of length at least $n+1$ such that its prefix of length $n$ and its suffix of length $n$ are both special (left or right) and
no interior factor of length $n$ is special.
We can now define the \textit{reduced Rauzy graph of order $n$ of $\u$} $\Gamma_n'(\u)$.
Its set of vertices equals the set of special (right or left) factors of $\L[n]$.
There is an edge $e$ from $v_1$ to $v_2$ if there is a $n$-simple path $p$ which begins with $v_1$ and ends with $v_2$.
Note that if for a given length $N$ there is no special factor,
the word $\u$ is periodic, i.e., can be written as $\u = v^{\omega} = vvv \ldots$, where $v$ is a finite word.
In such cases the definition of reduced Rauzy graph of order $n$ results in an empty graph for all $n \geq N$.

Consider now an infinite word $\u$ with language closed under reversal.
We can define the \textit{super reduced Rauzy graph of order $n$ of $\u$} denoted $\Gamma_n''(\u)$ (first defined in \cite{BuLuGlZa}).
Its set of vertices is formed by unordered pairs $\left \{ w, \overline{w} \right \}$,
where $w \in \L[n]$ is a special factor.
There is an undirected edge $\left \{ e, \overline{e} \right \}$ from $\left \{ v, \overline{v} \right \}$ to $\left \{ w, \overline{w} \right \}$ if there is an $n$-simple path $e$ starting with $v$ or $\overline{v}$ and ending with $w$ or $\overline{w}$.
Note that $\Gamma_n''(\u)$ may have multiple edges or loops.

\subsection{Palindromic richness}

We say that a finite word $w$ is \textit{rich} if it has exactly $|w|+1$ palindromic factors.
It is in fact the maximum number of palindromic factors possible.
An infinite word is rich if all its factors are rich.
This definition arises from the work done in \cite{DrJuPi} and in \cite{GlJuWiZa}.

The following theorem summarizes some known characterizations of rich words.
Listed equivalent definitions of richness have been proven respectively in~\cite{GlJuWiZa}, \cite{BuLuGlZa}, \cite{BuLuGlZa2}.
\begin{thm}[\cite{GlJuWiZa}, \cite{BuLuGlZa}, \cite{BuLuGlZa2}] \label{equiv_rich}
For any infinite word $\u$ the following conditions are equivalent:
\begin{enumerate}
\item $\u$ is rich,
\item any complete return word to a~palindromic factor of $\u$ is a~palindrome,
\item for any factor $w$ of $\u$, every factor of $\u$ that contains $w$ only as its prefix and $\overline w$ only as its suffix is a~palindrome,
\item each factor of $\u$ is uniquely determined by its longest palindromic prefix and its longest palindromic suffix.
\end{enumerate}
\end{thm}

The following theorem gives a characterization of rich infinite words with language closed under reversal.
It also answers the question whether the inequality \eqref{eq:norm_ineq} in \Cref{nerovnost-palindromy} is reached for rich words.
\begin{thm}[\cite{BuLuGlZa}] \label{rich_opulent}
Let $\u$ be an infinite word with language $\L$ closed under reversal.
Then $\u$ is rich if and only if $\PCn{n+1} + \PCn{n} = \Delta \FC{n} + 2$ for all $n \in \N$.
\end{thm}

In \cite{BaPeSta2} a different characterization of rich words with language closed under reversal has been shown.
It is based on bispecial orders of factors.

Let us mention two classes of words that we will refer to later and that are also rich.
The well-known \textbf{Sturmian} words are exactly those words with factor complexity $\FC{n} = n + 1$ for all $n \in \N$.
The second class consists of so-called \textbf{episturmian} words.
They are defined as words whose language is closed under reversal and have at most one left special factor of every length.

For further information on episturmian or rich words you can refer for instance to \cite{Be_survey,DrJuPi,GlJuWiZa,GlJu}.
A~differently constructed class of rich words than the two mentioned above may be found in \cite{AmFrMaPe}.

\subsection{{$\Theta$-palindromicity}}

The reversal mapping can in fact be understood as a special case of a more general mapping $\Theta$: $\A^* \mapsto \A^*$ which
satisfies two conditions:
\begin{enumerate}
  \item it is an involution, i.e., $\Theta^2 = {\rm id}$;
  \item $\Theta$ is an antimorphism, i.e., $\forall v,w \in \A^*$, $\Theta(vw) = \Theta(w)\Theta(v)$.
\end{enumerate}
From now on, $\Theta$ will denote a mapping with such properties.
It is easy to verify that $\Theta$ can only be a composition of the reversal mapping and an involutory permutation of letters.
Such permutation has two types of cycles distinguished by their length $1$ or $2$.

The definitions in the previous section can be modified in the following way.
A finite word $w$ is called a \textit{$\Theta$-palindrome} if $w = \Theta(w)$.
\textit{$\Theta$-palindromic complexity} is a mapping $\N \mapsto \N$ defined as
$$
\PC{n} := \# \left \{ w \in \L[n] \mid w = \Theta(w) \right \}.
$$
The language of $\u$ is \textit{closed under $\Theta$} if for any factor $w \in \L$ we have $\Theta(w) \in \L$.
The definition of super reduced Rauzy graph is analogous: for words closed under $\Theta$ the reversal mapping is replaced by $\Theta$.

As we will see later, a~crucial difference between the reversal mapping and a $\Theta$ mapping different from the reversal mapping
 is that there exists a letter $a$ such that $a \neq \Theta(a)$.
Therefore there exist words that have no non-empty $\Theta$-palindromic suffix at all.

The research on $\Theta$-palindromicity has led to several generalizations of known results.
We will recall one such generalization.
A word is  \textbf{$\Theta$-episturmian} if its language is closed under $\Theta$
and if for each $n$ there exists at most one left special factor of length $n$.
Such words are discussed for instance in \cite{BuLuLuZa}.
In \cite{AnZaZo} these words are called pseudopalindromic.
A more general class that we will mention later is discussed in \cite{BuLuLuZa2}.

\section{$\Theta$-palindromic and factor complexity}
\label{sec:pcfc}

As for the reversal mapping, if $\u$ has its language closed under $\Theta$, then it is recurrent.
To see that, consider a non-$\Theta$-palindromic factor $w \in \L$.
As $\L$ is closed under $\Theta$ we can find an~occurrence of $\Theta(w)$ and so find a factor $r$
 having $w$ as a prefix (respectively suffix) and $\Theta(w)$ as a suffix (respectively prefix).
If $r$ is not a $\Theta$-palindrome, then we can find a~different occurrence of $\Theta(r)$ and in it a second occurrence of $w$.
If $w$ or $r$ is a $\Theta$-palindrome it suffices to consider a non-$\Theta$-palindromic factor that contains $w$ or $r$ as factor.
Such factor either exists (for instance one can look only among prefixes of $\u$) or $\u$ is trivial (periodic with period $1$).

The following observation, stated as a lemma, will be used in the proof of \Cref{nerovnost}.
We provide a short proof.

\begin{lem}\label{ref_req_explain}
Let $\u$ be an infinite word with language closed under $\Theta$.
If there exists an integer $N$ such that there is no special factor of length $N$,
then for all $n \geq N$ we have $\PC{n} + \PC{n+1} = 2$.
\end{lem}

\begin{proof}
As $\u$ is recurrent, it is clear that for all $n\geq N$, $\Gamma_n(\u)$ is a cycle.
Take $n \geq N$.
We will consider an undirected graph $G$ whose vertices are defined by unordered pairs $\{w,\Theta(w)\}$ for all $w \in \L[n]$.
Let $e$ be an edge of $\Gamma_n(\u)$ going from the vertex $w$ to the vertex $v$.
For each pair $\{e,\Theta(e)\}$, there is an edge in $G$ between the vertex $\{w,\Theta(w)\}$ and the vertex $\{v,\Theta(v)\}$.
Note that there may be multiple edges and also loops (in the case $e = \Theta(e)$).
According to the terminology already used this graph may be called ``super Rauzy graph of order $n$''.

As $\Gamma_n(\u)$ is a cycle, it is clear that $G$ is formed by one elementary path $p$ and at most $2$ loops - one possibly on the first vertex of $p$, one possibly on the last vertex of $p$.
Suppose there is an interior vertex $\{w,\Theta(w)\}$ in $p$ such that $w = \Theta(w)$.
Then $w$ would be a special factor of length $n$ and we have a contradiction.
Denote by $\{s,\Theta(s)\}$ the first vertex of $p$.
If there is a loop on $\{s,\Theta(s)\}$, then it is clear that $s \neq \Theta(s)$ as otherwise $s$ would be special.
On the other hand, if there is not a loop on $\{s,\Theta(s)\}$, then $s = \Theta(s)$ as $\Gamma_n(\u)$ would be disconnected.
The reasoning is analogous for the last vertex of $p$.

As a loop in $G$ is in fact a $\Theta$-palindrome of length $n+1$,
we have at most $2$ $\Theta$-palindromes of length $n$ or $n+1$ - one associated with the first vertex of $p$ and one with the last vertex of $p$.

\end{proof}

We will now prove the analogue of \Cref{nerovnost-palindromy} for languages closed under $\Theta$.

\begin{proof}[Proof of \Cref{nerovnost}]
Consider the operation $\theta$ which to every vertex $w$ of a Rauzy graph associates $\Theta(w)$ and to every edge $e$ associates $\Theta(e)$.
Since we supposed the language closed under $\Theta$, $\theta$ is well defined on all Rauzy graphs of $\u$.
In fact, $\theta$ maps every Rauzy graph $\Gamma_n (\u)$ onto itself.
We will omit $\u$ in the notation of Rauzy graphs from now on.

Fix $n$.
We are interested in $n$-simple paths.
As $\L$ is closed under $\Theta$, $\u$ is recurrent.
Thus the graph $\Gamma_n$ is strongly connected and every vertex and every edge belongs to an $n$-simple path.

For an edge satisfying $\theta(e) = e$ we find the $n$-simple path $p$ which contains $e$.
Then the operation $\theta$ must map the path represented by $p$ onto itself, i.e., $p = \Theta(p)$.
Similarly, for a non-special vertex $w$ such that $\theta(w) = w$, the $n$-simple path containing $w$ is mapped by $\theta$ onto itself.

Note that an $n$-simple path can contain at most one $\Theta$-palindrome of length $n$ or $n+1$ which is not its prefix or suffix.
To prove that suppose the contrary.
Let $z$ be an $n$-simple path containing two interior $\Theta$-palindromic factors of length $n$ or $n+1$.
Denote $p_1$ the $\Theta$-palindrome with leftmost occurrence in $z$ which is not a prefix of $z$.
Since there is another $\Theta$-palindrome of length $n$ or $n+1$ and $z$ is a $\Theta$-palindrome,
there is another occurrence of $p_1$ as an interior factor of $z$.
Let $r$ be the shortest left special factor such that $p_1$ is its suffix.
If there is no occurrence of $r$ in $z$,
 then there is another occurrence in $z$ of the prefix of $z$ of length $n$ which is neither its prefix nor suffix.
If there is at least one occurrence of $r$ in $z$,
 then either the prefix of length $n$ of $r$ is a special factor of $z$ different from its prefix or suffix,
 or again the prefix of $z$ of length $n$ has another occurrence in $z$ different from its prefix or suffix.
In both cases, we can find an interior occurrence of a special factor of length $n$ - a contradiction with $z$ being an $n$-simple path.

To give an upper bound on $\PC{n} + \PC{n+1}$ therefore consists in finding the number of $n$-simple paths
 in $\Gamma_n$ that are mapped by $\theta$ onto themselves
and the number of special $\Theta$-palindromes of length $n$.
For that we can consider the reduced Rauzy graph $\Gamma_n'$.

The set of vertices of $\Gamma_n'$ can be partitioned into two disjoint subsets.
The first subset is given by special $\Theta$-palindromes of length $n$.
Let $\alpha$ denote the number of such vertices.
The second subset is formed by the rest of the vertices -- non-$\Theta$-palindromic special factors of length $n$.
Let the number of such vertices be denoted by $2 \beta$.
The total number of vertices in $\Gamma_n'$ is then $\alpha + 2 \beta$.

It is clear that the super reduced Rauzy graph $\Gamma_n''$ has exactly $\alpha+\beta$ vertices.
As $\Gamma_n''$ is connected, it has at least $\alpha + \beta - 1$ edges.
Thus the number of edges in $\Gamma_n'$ is at least $2(\alpha + \beta - 1)$.
These edges correspond to the $n$-simple paths in $\Gamma_n$ which are not mapped by $\theta$ onto themselves.

As already mentioned, the number of $\Theta$-palindromes of length $n$ or $n+1$ is bounded by the number of $n$-simple paths in $\Gamma_n$ which are mapped onto themselves,
and the number of special $\Theta$-palindromic vertices.
We thus have
$$
\PC{n} + \PC{n+1} \leq \sum_{w \text{ is special}} {\rm deg}_{+}(w) - 2(\alpha + \beta - 1) + \alpha,
$$
where the first summand is the number of all $n$-simple paths in $\Gamma_n$,
the second summand estimates the number of $n$-simple paths which are not mapped onto themselves,
and the third counts the number of special $\Theta$-palindromic vertices.
We obtain
\begin{align} \label{last_eq}
\PC{n} + \PC{n+1} & \leq &  \displaystyle\sum_{w \in \L[n] \text{ is special}} {\rm deg}_{+}(w) - (\alpha + 2 \beta) + 2 = \notag \\
 & = & \displaystyle\sum_{w \in \L[n] \text{ is special}} \left( {\rm deg}_{+}(w) - 1 \right) + 2 =  \\
 & = & \Delta \FC{n} + 2. \notag
\end{align}

The proof is finished except for the special case where the word $\u$ is periodic.
This implies existence of $N$ such that for all $n \geq N$ we have $\Delta \FC{n} = 0$, i.e., there is no special factor of length $n$.
It this case the reduced Rauzy graph of order $n$ is empty.
 According to \Cref{ref_req_explain} we have $\PC{n} + \PC{n+1} = 2$ for all $n \geq N$, which completes the proof.

\end{proof}

\begin{coro}
\label{pseudorovnost}
Let $ \L $ be closed under $\Theta$.
Let $n \geq 0$.
The equality
\begin{equation}\label{eq:pseudorovnost}
\PC{n} + \PC{n+1} = \Delta \FC{n} + 2
\end{equation}
holds if and only if both of the following conditions are satisfied:
\begin{enumerate}
  \item the graph $\Gamma_n''$ after removing loops is a tree,
  \item any simple path forming a loop in the graph $\Gamma_n''$ is mapped by $\theta$ onto itself.
\end{enumerate}
\end{coro}

\begin{proof}

First, we rewrite the equality in the relation \eqref{last_eq} as
$$
\Delta \FC{n} + 2 = \# {\text{$n$-simple paths in $\Gamma_n'$}} - \alpha - 2\beta + 2
$$
where $\alpha$ is again the number of special $\Theta$-palindromes of length $n$
and $2 \beta$ is the number of non-$\Theta$-palindromic special factors of length $n$.
Let $p$ denote the number of $n$-simple paths that are mapped by $\theta$ onto themselves
 and $m$ the number of $n$-simple paths that are not mapped onto themselves.

We will prove the equivalence directly.
Suppose that $\PC{n} + \PC{n+1} = \Delta \FC{n} + 2$.
We have
$$
\alpha + p = \PC{n} + \PC{n+1} = \Delta \FC{n} + 2 = p + m - \alpha - 2\beta + 2.
$$
Therefore
$m = 2 (\alpha + \beta - 1)$.
From the proof of \Cref{nerovnost}, this is equivalent with having
 exactly the minimum number of non-$\Theta$-palindromic $n$-simple paths needed so that the graph $\Gamma_n'$ is strongly connected.
As loops are not needed to have the graph strongly connected, they are counted in the number $p$.
These two facts are equivalent with $\Gamma_n''$ being a tree after removing loops and all loops being $\Theta$-palindromic.

\end{proof}

A natural question is whether there exist such words for which the equality \eqref{eq:pseudorovnost} is satisfied for all $n$.
The answer is positive only for $\Theta$ equal to the reversal mapping.
In any other case we have $\PC{1} < \# \A$ and thus
$$
\PC{0} + \PC{1} < \left ( \# \A - 1 \right ) + 2 = \Delta \FC{0} + 2.
$$
However, there are words saturated in $\Theta$-palindromes that satisfy the equality for all $n \geq 1$.
Examples of such words are given in Section \ref{sec:examples}.

\section{$\Theta$-Richness} \label{sec:pseudorich}

Recall that a finite word is rich if it has the maximum possible number of palindromic factors.
If we want to use the same notion of richness for $\Theta$-palindromes, we need to take into consideration that there is always fewer $\Theta$-palindromes since some letters may not be $\Theta$-palindromes.
To define richness we need to evaluate the number of these letters contained in a given word.
We introduce the set $\gamma(w)$ as follows
$$
\gamma(w) :=   \big\{ \{a,\Theta(a)\} \mid a \in \A, a \neq \Theta(a), a \text{ or } \Theta(a) \in {\mathcal L}(w) \big\}.
$$

It is clear that if $w'$ is a prefix of some finite word $w$, then $\# \gamma(w') \leq \# \gamma(w)$.
It is also clear that if for some letter $x \in \A$ we have $\gamma(w) \neq \gamma(wx)$, then the word $wx$ has no $\Theta$-palindromic suffix as $\Theta(x)$ does not occur in it.
Also, as $\gamma(wx) = \gamma(w) \cup  \big \{ \left \{x,\Theta(x) \right \}  \big \}$, one can see that in this case $\# \gamma(wx) = \# \gamma(w) + 1$.

The following proposition gives the maximum number of $\Theta$-palindromes that may occur in a finite word.

\begin{prop}
\label{w_plus_jedna}
Let $w$ be a finite word.
Then the number of its $\Theta$-palindromic factors is bounded by $|w| + 1 - \# \gamma(w)$.
\end{prop}

\begin{proof}
The proof is done by induction on $|w|$.
The statement is clear for $|w| \leq 1$.
Suppose that for $|w| \leq N$, $w$ contains at most $|w| + 1 - \# \gamma(w)$ different $\Theta$-palindromes.
Let $x \in \A$.
We will count the number of $\Theta$-palindromes in the word $wx$.
It is clear that the number of $\Theta$-palindromes in $wx$ is equal to
the number of $\Theta$-palindromes in $w$ plus
the number of new $\Theta$-palindromic factors -- $\Theta$-palindromic suffixes of $wx$ that do not occur in $w$.

First suppose we have $\gamma(w) \neq \gamma(wx)$.
As there is no $\Theta$-palindromic suffix
the number of $\Theta$-palindromes in $wx$ equals the number of $\Theta$-palindromes in $w$ which is at most $|w| + 1 - \# \gamma(w) = |wx| +1 - \# \gamma(wx)$.

Suppose $\gamma(wx) = \gamma(w)$.
Suppose for contradiction that $wx$ contains more than $|wx| + 1 - \# \gamma(wx)$ different $\Theta$-palindromes.
As $|wx| + 1 - \# \gamma(wx) = \left( |w| + 1 - \# \gamma(w) \right) + 1$,
it is clear that there must be at least two different $\Theta$-palindromes $p$ and $q$ such that $w = yp$ and $w = zq$ that do not occur previously in $w$ (as they were not yet counted).
We can assume that $|p| > |q|$ and $p$ is the longest of such suffixes of $wx$.
On the other hand we have $p = vq = q\Theta(v)$ where $|v| > 0$.
Therefore every such factor $q$ occurs already in $w$ and has already been counted in the number of $\Theta$-palindromes in $w$ -- a contradiction.
We can see that the number of $\Theta$-palindromes in $wx$ is equal to either the number of $\Theta$-palindromes in $w$ or the number of $\Theta$-palindromes in $w$ plus $1$.
Therefore the upper bound is $|w| + 1 - \# \gamma(w) + 1$ which equals $|wx| + 1 - \# \gamma(wx)$.
\end{proof}

We can see from the preceding proof that if there is a new $\Theta$-palindromic factor,
it is the longest $\Theta$-palindromic suffix.
We will now define a class of words having maximum number of $\Theta$-palindromes.

\begin{defi}
A finite word $w$ is {\bf $\Theta$-rich}
if it contains $|w| + 1 - \# \gamma(w)$ $\Theta$-palindromic factors.

An infinite word is {\bf $\Theta$-rich}
if all its factors are $\Theta$-rich.
\end{defi}

One can see that all factors of a finite $\Theta$-rich word are $\Theta$-rich as well.
Thus, if an infinite word has all its prefixes $\Theta$-rich,
then all its factors are $\Theta$-rich and it is by definition $\Theta$-rich.

The following equivalence is a consequence of the last remark and of the proof of \Cref{w_plus_jedna}.
It is a $\Theta$-rich analogue of results in \cite{DrJuPi,GlJuWiZa}.

\begin{prop} \label{ups}
An infinite or finite word $\u$ is $\Theta$-rich
if and only if
the longest $\Theta$-palindromic suffix of each prefix $p$ is unioccurrent
except for prefixes having the form $px$,
with $x \in \A$ and $\gamma(p) \neq \gamma(px)$.
\end{prop}

\Cref{ups} leads to the following corollary, which is a $\Theta$-rich analogue of a result in \cite{BuLuGlZa}.

\begin{coro} \label{rich_slova_alternuji}
Let $\u$ be an infinite $\Theta$-rich word.
Then for all $w \in \L$, $w \neq \Theta(w)$, the occurrences of $w$ and $\Theta(w)$ alternate.
\end{coro}

\begin{proof}
Suppose there exist a factor $w \in \L$, $w \neq \Theta(w)$,
and a factor $v \in \L$, $|v| > |w|$, such that $w$ is its prefix and its suffix and there is no occurrence of $\Theta(w)$ in $v$.
Since the last letter of $v$ already occurred in $v$
and $v$ does not have a unioccurrent $\Theta$-palindromic suffix,
according to \Cref{ups} we have a contradiction.
\end{proof}

The following proposition gives an equivalent characterization of $\Theta$-rich words.

\begin{prop} \label{rich_nece}
Let $\u$ be an infinite or finite word.
Then $\u$ is rich $\Theta$-rich
if and only if for each factor $v \neq \varepsilon$ of $\u$,
any factor of $\u$ beginning with $v$ and ending with $\Theta(v)$,
containing no other occurrences of $v$ or $\Theta(v)$, is a $\Theta$-palindrome;
and for all $a \in \A$, the occurrences of $a$ and $\Theta(a)$ in $\u$ alternate.
\end{prop}

\begin{proof}
($\Rightarrow$):
Let $v$ be a factor of $\u$.
Consider a factor $r$ containing $v$ as a prefix and $\Theta(v)$ as a suffix, with no other occurrences of $v$ or $\Theta(v)$.

If $|v| = |r|$, then $v = \Theta(v)$ and $r$ is a $\Theta$-palindrome.

Let us assume $|v| < |r|$.
Let $s$ be the longest $\Theta$-palindromic suffix of $r$.
(There is such a suffix as $r$ is $\Theta$-rich and the image by $\Theta$ of the last letter of $r$ has already occurred in $r$ as the first letter.)

If $|s| \leq |v|$ then we can write $\Theta(v) = zs$ with $z \in \A^*$.
Then $v = \Theta(s)\Theta(z) = s\Theta(z)$.
We have a second occurrence of $s$ in $r$ -- a contradiction with its unioccurrence.

If $|v| < |s| < |r|$ then as $\Theta(v)$ is a suffix of $s$ we can write $s = z\Theta(v) = v\Theta(z)$.
Therefore we can find a second occurrence of $v$ in $r$ -- a contradiction with choice of $r$.

The last possibility left is $s = r$ and thus $r$ is a $\Theta$-palindrome.

Alternation of $a$ and $\Theta(a)$ follows from \Cref{rich_slova_alternuji}.

($\Leftarrow$):
Suppose for contradiction that $\u$ is not $\Theta$-rich.
Then there is a prefix $p$ whose longest $\Theta$-palindromic suffix $s$ occurs at least twice in $p$
or
there is a prefix $p$ with no non-empty $\Theta$-palindromic suffix such that $p$ ends in a non-$\Theta$-palindromic letter $a$ and has at least one another occurrence of $a$ or $\Theta(a)$.

In the first case, %denote $s$ to be such a suffix and take the longest one.
as $s$ occurs in $p$ for a second time, there is a complete return word to $s$ contained in $p$ which is its suffix.
The complete return word is a $\Theta$-palindrome and is longer than $s$ -- a contradiction.

In the second case, as occurrences of $a$ and $\Theta(a)$ alternate, we can find the rightmost occurrence of $\Theta(a)$ in $p$.
The suffix starting at that occurrence is then a word beginning with $\Theta(a)$, ending with $a$ and with no other occurrences of $a$ or $\Theta(a)$.
We supposed such words were $\Theta$-palindromic.
Thus $p$ has a $\Theta$-palindromic suffix and we have a contradiction.

\end{proof}

In the last proof we did not need all the assumptions to prove the second implication.
We can formulate a weaker sufficient condition as follows.

\begin{prop} \label{rich_sufi}
Let $\u$ be an infinite or finite word satisfying
\begin{itemize}
  \item all complete return words to $\Theta$-palindromes are $\Theta$-palindromes;
  \item for any  $a \in \A$, $a \neq \Theta(a)$, occurrences of $a$ and $\Theta(a)$ alternate;
  \item factors beginning with $a \in \A$ and ending with $\Theta(a)$, with no other occurrences of $a$ or $\Theta(a)$, are $\Theta$-palindromes.
\end{itemize}
Then $\u$ is $\Theta$-rich.
\end{prop}

The following lemma will be used in the proof of \Cref{pseudo_rovnost_rich}.

\begin{lem} \label{P_acka_alternuji}
Let $\u$ be an infinite word with language closed under $\Theta$.
If $\PC{1} + \PC{2} = \Delta \FC{1} + 2$
then for all letters $a$, $a \neq \Theta(a)$, occurrences of $a$ and $\Theta(a)$ alternate.
\end{lem}

\begin{proof}
Let $a \in \A$ and suppose $a$ is not a $\Theta$-palindrome.
We will look at the super reduced Rauzy graph $\Gamma_1''$ for which the statements in \Cref{pseudorovnost} hold.

Suppose $a$ is not special.
Then $a$ is a factor of exactly one $1$-simple path $p$.
If the edge $\left \{ p, \Theta(p) \right \}$ is a loop in $\Gamma_1''$ then the occurrences of $a$ and $\Theta(a)$ must alternate as $p = \Theta(p)$.
If $\left \{ p, \Theta(p) \right \}$ is not a loop then there is no path in $\Gamma_1''$ that would contain $2$ occurrences of $p$ and no occurrence of $\Theta(p)$.
Such path would form a cycle in $\Gamma_1''$ which is not possible.

Suppose now that $a$ is special.
Then $\left \{ a, \Theta(a) \right \}$ is a vertex of $\Gamma_1''$.
We will look at paths going from $\left \{ a, \Theta(a) \right \}$ again to $\left \{ a, \Theta(a) \right \}$
 which are not going through the vertex $\left \{ a, \Theta(a) \right \}$.
There are $2$ possibilities.

In the first case such path corresponds to a loop in $\Gamma_1''$ connected to $\left \{ a, \Theta(a) \right \}$.
Therefore, the corresponding path in $\Gamma_1$ is a $1$-simple path and is a $\Theta$-palindrome beginning with $a$ or $\Theta(a)$ respectively and ending with $\Theta(a)$ or $a$ respectively.

In the second case such path  starts with an edge $\left \{ e,\Theta(e) \right \}$.
It is clear that the end of this path is again the same edge $\left \{ e,\Theta(e) \right \}$
 as otherwise there would be a cycle in $\Gamma_1''$.
 If the corresponding path in $\Gamma_1$ starts with $e$, it must end with $\Theta(e)$ as otherwise there would be a multiple edge (forming a cycle) between the vertices of $\Gamma''$ connected by $\left \{ e,\Theta(e) \right \}$.
Therefore, if such path starts with $a$, it must end with $\Theta(a)$ and the proof is finished.

\end{proof}

In what follows, to ease the notation, we will borrow some elements from the free group generated by $\A$ and use them in the following way.
Let $w$ be a finite word.
By $w^{-1}$ we denote an element whose concatenation with $w$ produces the empty word, i.e., $ww^{-1} = w^{-1}w = \varepsilon$.
The notation is justified by the fact that we will never concatenate $w^{-1}$ with anything else than $w$.

\begin{proof}[Proof of \Cref{pseudo_rovnost_rich}]

~

\vspace{\baselineskip}

($\Leftarrow$):

Suppose that $\u$ is not rich.
According to \Cref{P_acka_alternuji} non-$\Theta$-palindromic letters alternate.
\Cref{rich_nece} implies there exists a factor $v \in \L$ and a factor $r \in \L$
such that $r$ begins with $v$ and ends with $\Theta(v)$,
with no other occurrences of $v$ or $\Theta(v)$, and $r$ is not a $\Theta$-palindrome.
Then $r$ is of the form
$$
r = vwxsy\Theta(w)\Theta(v),
$$
where $w$ and $s$ are finite words and $x,y \in \A$, $x \neq \Theta(y)$.
As the language is closed under $\Theta$, we can see that $vw\Theta(y) \in \L$.
This means that the factor $vw$ is right special and therefore $\left \{ vw, \Theta(vw) \right \}$ is a vertex in $\Gamma_{|vw|}''$.

If $r$ is a $|vw|$-simple path then $r$ forms a loop in $\Gamma_{|vw|}''$.
Therefore according to \Cref{pseudorovnost} it is a $\Theta$-palindrome -- a contradiction.

Suppose now that $r$ is not a $|vw|$-simple path, i.e., contains other special factors of length $|vw|$.
Denote by $p,q$ two factors of $r$ such that they are $|vw|$-simple paths and $p$ is a prefix and $q$ a suffix of $r$.
Let $p'$ denote the special suffix of length $|vw|$ of $p$ and $q'$ the special prefix of $q$ of length $|vw|$.
It is clear that vertices $\left \{ vw, \Theta(vw) \right \}$ and $\left \{ p', \Theta(p') \right \}$ are connected
as well as $\left \{ vw, \Theta(vw) \right \}$ and $\left \{ q', \Theta(q') \right \}$ are connected.

If $p'$ and $q'$ coincide, i.e., $r = p{p'}^{-1}q$, then as $x \neq \Theta(y)$ we have a multiple edge in $\Gamma_{|vw|}''$ -- a cycle in $\Gamma_{|vw|}''$ and thus a contradiction.
If $p'$ and $q'$ do not coincide there exists a path in $\Gamma_{|vw|}''$ leading from $\left \{ p', \Theta(p') \right \}$ to $\left \{ q', \Theta(q') \right \}$,
 avoiding $\left \{ vw, \Theta(vw) \right \}$,
 which together with the two paths connecting $\left \{ vw, \Theta(vw) \right \}$ forms a cycle and we have again a contradiction with \Cref{pseudorovnost}.

\vspace{\baselineskip}

($\Rightarrow$):

We say that a factor $r$ is a \textit{realization} of an edge in $\Gamma''$ between vertices $\left \{ v, \Theta(v) \right \}$ and $\left \{ w, \Theta(w) \right \}$
if
\begin{itemize}
\item  $r$ has a prefix $v$ or $\Theta(v)$ and a suffix $w$ or $\Theta(w)$ OR
\item  $r$ has a prefix $w$ or $\Theta(w)$ and a suffix $v$ or $\Theta(v)$,
\end{itemize}
and contains no interior occurrences of factors $v$, $\Theta(v)$, $w$ and $\Theta(w)$.

We will denote by $v \! \rightharpoondown \! w$ a realization of an edge between vertices $\left \{ v, \Theta(v) \right \}$ and $\left \{ w, \Theta(w) \right \}$
which starts with $v$ and ends with $w$.

First, we are going to show that for every pair of distinct vertices $\left \{ v, \Theta(v) \right \}$ and $\left \{ w, \Theta(w) \right \}$
there are exactly two realizations of an edge between them and that they are $\Theta$ images of each other.
Let $w \neq v$ and $w \neq \Theta(v)$.
Suppose without loss of generality there is a realization $v \! \rightharpoondown \! w$.
We will find two consecutive realizations, denoted $r_1$ and $r_2$. See \Cref{fig:duk_o1}.

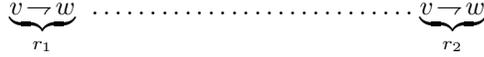
\begin{figure}[ht]
\begin{center}
\begin{picture}(200,15)

\put(20,3){\makebox[180pt]{$\underbrace{v \! \rightharpoondown \! w}_{r_1}$ \dotfill $\underbrace{v \! \rightharpoondown \! w}_{r_2}$}}

\end{picture}
\end{center}
\caption{Two consecutive realizations $v \! \rightharpoondown \! w$.}
\label{fig:duk_o1}
\end{figure}

According to \Cref{rich_slova_alternuji}
occurrences of each factor alternate with occurrences of its $\Theta$ image.
Therefore we can find occurrences of $\Theta(v)$ and $\Theta(w)$ between occurrences of $r_1$ and $r_2$.
Let us first find the leftmost occurrence (after $r_1$) of $\Theta(v)$.
As every factor beginning with $v$ and ending with $\Theta(v)$ is a $\Theta$-palindrome, we can find a realization $\Theta(w) \! \rightharpoondown \! \Theta(v)$.
Denote this realization by $q_1$.
See \Cref{fig:duk_o2}.
Note that the occurrences of the factor $w$ in $r_1$ and $\Theta(w)$ in $q_1$ may coincide.

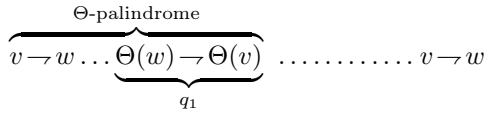
\begin{figure}[ht]
\begin{center}
\begin{picture}(200,15)

\put(20,3){\makebox[180pt]{$\overbrace{v \! \rightharpoondown \! w \dots \underbrace{\Theta(w) \! \rightharpoondown \! \Theta(v)}_{q_1}}^{\Theta\text{-palindrome}}$ \dotfill $v \! \rightharpoondown \! w$}}

\end{picture}
\end{center}
\caption{The leftmost occurrence of realization $\Theta(w) \! \rightharpoondown \! \Theta(v)$.}
\label{fig:duk_o2}
\end{figure}

Let us now find the rightmost occurrence (before $r_2$) of $\Theta(w)$.
Analogously, we find a realization $\Theta(w) \! \rightharpoondown \! \Theta(v)$.
Denote this realization by $q_2$.
See \Cref{fig:duk_o3}.

\begin{figure}[ht]
\begin{center}
\begin{picture}(200,15)

\put(20,3){\makebox[180pt]{$v \! \rightharpoondown \! w$ \dotfill $\overbrace{\underbrace{\Theta(w) \! \rightharpoondown \! \Theta(v)}_{q_2} \ldots v \! \rightharpoondown \! w}^{\Theta\text{-palindrome}}$}}

\end{picture}
\end{center}
\caption{The rightmost occurrence of realization $\Theta(w) \! \rightharpoondown \! \Theta(v)$.}
\label{fig:duk_o3}
\end{figure}
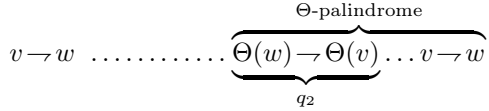

Suppose now that occurrences of $q_1$ and $q_2$ are not equal.
If they are not consecutive, we can find a realization of $\Theta(w) \! \rightharpoondown \! \Theta(v)$, denoted by $q_1'$, such that $q_1'$ and $q_2$ are consecutive.
Now, using the same arguments as for consecutive realizations $r_1$ and $r_2$, we can find a realization $v \! \rightharpoondown \! w$ between $q_1'$ and $q_2$.
This is a contradiction with $r_1$ and $r_2$ being consecutive.
Thus the occurrence of $q_1$ equals the occurrence $q_2$ and $q_1 = q_2$.
$\Theta$-palindromes depicted in \Cref{fig:duk_o2} and \Cref{fig:duk_o3} imply $q_1 = \Theta(r_1)$ and $q_2 = \Theta(r_2)$, respectively.
Therefore $r_1 = r_2$ and we deduce that every realization $v \! \rightharpoondown \! w$ equals $r_1$
and every realization $\Theta(v) \! \rightharpoondown \! \Theta(w)$ equals $\Theta(r_1)$.

So far we have shown that for every pair of distinct vertices $\left \{ v, \Theta(v) \right \}$ and $\left \{ w, \Theta(w) \right \}$
there are exactly two realizations of an edge between them and that they are $\Theta$ images of each other.
This implies that $\Gamma''$ after removing loops is a tree.
\Cref{rich_slova_alternuji} claims that occurrences of $w$ and $\Theta(w)$ alternate,
 thus any realization between the same vertices $\{w, \Theta(w)\}$ is of the form $w \! \rightharpoondown \! \Theta(w)$ or $\Theta(w) \! \rightharpoondown \! w$. By \Cref{rich_nece}, such realizations are $\Theta$-palindromes.
These realizations are in fact loops in $\Gamma''$.
Using \Cref{pseudorovnost} these two parts together imply
that the equality \eqref{eq:pseudorovnost} holds.

\end{proof}

\section{Examples of $\Theta$-rich words}\label{sec:rich_examples}

In this section we will provide examples of $\Theta$-rich words.
We first concentrate on a~generalization of episturmian words, namely $\Theta$-episturmian words,  mentioned earlier.
The following definition will serve to reveal the structure of the language of a $\Theta$-episturmian word.
Its main aim is to ease the notation.

\begin{defi}
Let $\u$ be a $\Theta$-episturmian word.
Denote by $\left( w_k \right)_{k=0}^{K}$ the sequence of all bispecial factors of $\u$,
 ordered by increasing length and starting by $w_0 = \varepsilon$ with $K$ possibly equal to $+\infty$.
Furthermore let $z_k$ denote the unique letter such that $w_kz_k$ is a prefix of $w_{k+1}$.
Denote by $p_k \in \A^*$ the word such that $w_{k+1} = p_k\Theta(z_k)w_k$.
For all $k \in \N$ and all $a \in \A$ let $r_{k,a}$ denote the complete return word to $w_k$ beginning with $w_ka$, if it exists.

\vspace{\baselineskip}

For a letter $a \neq \Theta(a)$, let $k_a$ denote the maximal index $k \leq K$, if it exists, such that
$w_k$ has non-$\Theta$-palindromic complete return words beginning with $w_ka$ and with $w_k\Theta(a)$.
If $k_a$ does not exist, we will consider it to be $+\infty$.
\end{defi}

Note that since there is at most one complete return word $w_k$ having a prefix $w_ka$, the definition of $r_{k,a}$ is correct.
As the language of $\u$ is closed under $\Theta$, $r_{k,a}$ exists if and only if $r_{k,\Theta(a)}$ exists.
Note also that $k_a \geq 1$ for every letter $a \neq \Theta(a)$.

\begin{prop}\label{prop:struktura_dle_ka}
Let $\u$ be a $\Theta$-episturmian word.
\begin{enumerate}
  \item \label{crw1} If $a = \Theta(a)$ and $r_{k,a}$ exists, then $r_{k,a}$ is a $\Theta$-palindrome;
  \item \label{crw2} if $a \neq \Theta(a)$, then  \begin{enumerate}[a.]
    \item \label{crw2a} if $k \leq k_a$, then $r_{k,a}$ equals $w_{k}aw_{k}$, furthermore $a$ is not contained in $w_k$ or in any other complete return word to $w_k$ except for $r_{k,a}$;
    \item \label{crw2b} if $k > k_a$, then at most one of the words $r_{k,a}$ and $r_{k,\Theta(a)}$ exist.% and if it exists, it is a $\Theta$-palindrome.
  \end{enumerate}
\end{enumerate}

\end{prop}

\begin{proof}
The proof will be done by induction on $k$.

For $k=0$, we have $w_0 = \varepsilon$ and the claim holds trivially (for non-$\Theta$-palindromic $a$ we are in the case (\ref{crw2a})).

Suppose the claim holds for $k$.
We will prove it for $k+1$.
We need to consider especially the case $k+1 = k_a+1$ for all $a$ to be able to assume also (\ref{crw2b}).

\vspace{\baselineskip}
Case (\ref{crw1}):

Let $a = \Theta(a)$.
If $a = z_k$ then $r_{k+1,a} = p_k\Theta(z_k)w_kas$ for some word $s$.
As there is at most one complete return word to $w_k$ beginning with $w_ka$,
we have $w_kas = w_ka\Theta(p_k)$.
Thus, $r_{k+1,a} = p_kaw_ka\Theta(p_k)$ which is a $\Theta$-palindrome.

If $a \neq z_k$  suppose that $r_{k+1,a}$ is not a $\Theta$-palindrome.
Therefore we can write $r_{k+1,a} = p_k\Theta(z_k)w_kax$ and $w_kax = yaw_kv = r_{k,a}v$ with $x,y,v \in \A^*$ .
By the assumption $r_{k,a}$ is a $\Theta$-palindrome which implies that $v$ cannot start with $z_k$ as $r_{k+1,a}$ would be a $\Theta$-palindrome.
In other words we can find a factor having the form $aw_kb$ with $b \neq z_k$.

We can then find two different factors $\Theta(z_k)w_ka \in \L$ and $\Theta(b)w_ka \in \L$.
Thus $w_ka$ is left special which is a contradiction with $w_kz_k$ being a second left special factor of the same length.
Thus $b = z_k$ and we have $r_{k+1,a} = p_k\Theta(z_k)r_{k,a}z_k\Theta(p_k)$ which is a $\Theta$-palindrome.
This concludes the proof of the case (\ref{crw1}).

\vspace{\baselineskip}
Case (\ref{crw2}):

Let now $a \neq \Theta(a)$. We distinguish the following three cases according to the value of $k_a$.
\begin{enumerate}[I.]
  \item $k+1 \leq k_a$

Suppose that $a$ is contained in $w_{k+1}$.
This means that $w_kaw_k$ is a factor of $w_{k+1}$ and so is $w_k\Theta(a)w_k$.
%We will distinguish three cases according to the value of $z_k$.

Suppose $z_k = a$.
We have $w_{k+1} = w_kaw_kpw_k\Theta(a)w_k$ where $p$ is a $\Theta$-palindrome or ${w_k}^{-1}$.
Then the complete return word $r_{k+1,a}$ equals $w_{k+1}aw_kpw_k\Theta(a)w_k$ which is a $\Theta$-palindrome and contradicts the assumption $k+1 < k_a$.

% Suppose now that $z_k = \Theta(a)$.
% We come to the same conclusion as in the previous case with $a$ and $\Theta(a)$ switched.

Suppose now $z_k \neq a$. % and $z_k \neq \Theta(a)$.
We can then find a factor of $w_{k+1}$ of the form $bw_kaw_kc$ with $b \neq \Theta(z_k)$.
(If $b = \Theta(z_k)$, then we would have a shorter bispecial factor than $w_{k+1}$ but longer than $w_k$.)
Since $r_{k+1,a}$ exists we can find a factor $\Theta(z_k)w_ka$.
Thus $w_ka$ is right special which contradicts $w_kz_k$ being right special of the same length.

We will now prove that $r_{k+1,a}$ is the only complete return word to $w_{k+1}$ to contain $a$.
Suppose that $a$ is contained in $r_{k+1,b}$ with $b \neq a$.
Since $a$ is not contained in $w_{k+1}$,  $r_{k+1,b}$ is of the form $w_{k+1}xw_{k+1}$ where $x \in \A^*$ contains $a$.
%This implies $b \neq z_k$.

We can find a factor of $x$ equal to $cw_ka$ with $c \neq \Theta(z_k)$.
As $r_{k+1,a}$ exists there is a factor $\Theta(z_k)w_ka$ which implies that $w_ka$ is right special.
Since especially $a \neq z_k$ as otherwise it would be contained in $w_{k+1}$, we arrive at a contradiction.

It remains to show that $r_{k+1,a} = w_{k+1}aw_{k+1}$.
Since $k+1 < k_a$ we know from the definition of $k_a$ that both $r_{k+1,a}$ and $r_{k+1,\Theta(a)}$ exist.
We can see that $r_{k+1,a} = p_k\Theta(z_k)r_{k,a}x$ with $x \in \A^*$.
If the first letter of $x$ is $z_k$ we are finished.
Suppose it is not true and denote that letter by $x_0 \neq z_k$.
We have $aw_kx_0 \in \L$.
As $r_{k+1,\Theta(a)}$ exists (from the definition of $k_a$) we know that $\Theta(z_k)w_k\Theta(a) \in \L$ which implies $aw_kz_k \in \L$.
This implies that $aw_k$ is right special.
As $\Theta(z_k) \neq a$ we can see that we have $2$ right special factors of the same length and thus $x_0$ must be $z_k$.
Finally as $r_{k,a} = w_kaw_k$ we have $r_{k+1,a} = p_k\Theta(z_k)r_{k,a}z_k\Theta(p_k) = w_{k+1}aw_{k+1}$.

  \item $k+1 = k_a+1$

If $w_{k+1}$ does not exist, then $r_{k+1,a}$ does not exist too.

If $w_{k+1}$ exists, suppose that both $r_{k+1,a}$ and $r_{k+1,\Theta(a)}$ exist and at least one of them is a $\Theta$-palindrome.
We can suppose without loss of generality that $r_{k,a}$ is a $\Theta$-palindrome.
We can then find a factor $\Theta(a)w_kz_k$.
$r_{k+1,\Theta(a)}$ contains a factor $\Theta(a)w_kb$ with $b \neq z_k$ as otherwise we would have two complete return words $r_{k+1,a}$.
The factor $\Theta(a)w_k$ is therefore right special which implies $z_k = a$.
As $\Theta(a)w_k\Theta(a) \in \L$ we have also $aw_ka \in \L$ and $aw_k$ is not right special.
Since $w_ka$ can only be extended as $w_kaw_k$ we arrive to a contradiction --
 $\u = pw_kaw_kaw_ka\ldots$ which implies $\u$ is not recurrent ($\Theta(a)$ occurs only in the word $p$).

  \item $k+1 > k_a+1$

It suffices to see that
if $r_{k,a}$ does not exist then $r_{k+1,a}$ does not exist too.

\end{enumerate}

\end{proof}

\begin{coro}\label{ka_exists}
For all letters $a \neq \Theta(a)$ the value of $k_a$ is finite.
\end{coro}

\begin{proof}
If $\u$ is periodic, then $K$ is finite and so is $k_a$.
For $\u$ aperiodic if $k_a = +\infty$ then we stay in the case (\ref{crw2a}) for all $k$.
In other words all bispecial factors do not contain $a$ which is clearly impossible.
\end{proof}

From the above corollary we can deduce that if $\Theta$ is different from the reversal mapping,
then there exists $k$ such that $\Delta \FC{|w_k|} < \# \A -1$, i.e., the growth of factor complexity is not constant.

The following corollary will serve to derive the criterion for richness of $\Theta$-episturmian words.

\begin{coro} \label{cor:zustavaji_palindromy}
Let $k > k_a$ and $r_{k,a}$ is a $\Theta$-palindrome.
Then either $r_{k+1,a}$ does not exist or is a $\Theta$-palindrome.
\end{coro}

\begin{proof}
Suppose that $r_{k+1,a}$ exists and is not a $\Theta$-palindrome.
Then $r_{k+1,a}$ begins by $p_k\Theta(z_k)r_{k,a}b$ where $z_k \neq b \in \A$.
As $\Theta(a)w_kb \in \L$ so is $\Theta(b)w_ka \in \L$.
Together with $\Theta(z_k)w_ka \in \L$ we see that $w_ka$ is left special, i.e., $z_k = a$.
This is impossible as we would have $r_{k+1,a} = p_k\Theta(a)w_ka\Theta(p_k)$ -- a $\Theta$-palindrome.
\end{proof}

\begin{coro}
Let $\u$ be a $\Theta$-episturmian word.
Let $l_k = |w_k|$.
If there exists an index $k_0$ such that $\PC{l_{k_0}} + \PC{l_{k_0}+1} = \Delta \FC{l_{k_0}} + 2$,
then for all $n \geq l_{k_0}$
$$\PC{n} + \PC{n+1} = \Delta \FC{n} + 2.$$
\end{coro}

\begin{proof}

First note that as $\u$ is $\Theta$-episturmian, $\Gamma_n''$ has at most one vertex for all $n$.

According to \Cref{pseudorovnost}, the graph $\Gamma_{l_{k_0}}''$ consists of one vertex
and $\PC{l_{k_0}} + \PC{l_{k_0}+1} - 1$ edges which are loops mapped onto themselves.
Therefore all complete return words to $w_{k_0}$ are $\Theta$-palindromes.
Thus, using \Cref{cor:zustavaji_palindromy}, for all $k>k_0$, the graph $\Gamma_{l_k}''$ consists
of one vertex with loops mapped onto themselves.
Thus using again \Cref{pseudorovnost}, we see that $\PC{l_k} + \PC{l_k+1} = \Delta \FC{l_k} + 2$.

We will prove that for all $k > k_0$, for all $n$ such that $l_{k-1} < n < l_k$, the graph $\Gamma_n''$
is of the same form as $\Gamma_{l_k}''$ (one vertex with loops mapped onto themselves)
and thus according to \Cref{pseudorovnost}
the equation $\PC{n} + \PC{n+1} = \Delta \FC{n} + 2$ holds.
It suffices to see that every $n$-simple path is in fact a $\Theta$-palindromic factor of a $(l_k)$-simple path or $w_k$.

It remains to see the case of finite value of $K$.
Then for all $n > K$, there is no special factor of length $n$.
According to \Cref{ref_req_explain}, we have $2 = \PC{n} + \PC{n+1} = \Delta \FC{n} + 2$.

\end{proof}

The following corollary stems from combining the previous claim with \Cref{pseudo_rovnost_rich}.
It gives a criterion to see whether a $\Theta$-episturmian word is $\Theta$-rich or not.

\begin{coro}
\label{criterion_t-epis}
$\Theta$-episturmian word is $\Theta$-rich if and only if $\PC{1} + \PC{2} = \Delta \FC{1} + 1$.
\end{coro}

\subsection{Examples}
\label{sec:examples}

This section contains examples of infinite words with properties mentioned above.

\begin{priklad} \label{slavny_pripad}
Let $\pi$ denote a morphism $\{ 0,1 \}^* \mapsto \{a, a', c\}^*$:
$$
\pi:  \left \{ \begin{array}{l} 0 \mapsto aa' \\ 1 \mapsto aa'c \end{array} \right. .
$$
Denote $\Theta$ an involutory antimorphism defined as
$$
\Theta: a \mapsto a', \ a' \mapsto a,  \ c \mapsto c.
$$

Then the infinite word $\u = \pi(\vv)$, where $\vv$ is a Sturmian word, has the following properties:
\begin{enumerate}
  \item is $\Theta$-episturmian,
  \item satisfies the equality $\PC{n} + \PC{n+1} = \Delta \FC{n} + 2$ for $n \geq 1$,
  \item is $\Theta$-rich.
\end{enumerate}

As $\vv$ is Sturmian, it is clear that $\u$ is aperiodic.
Furthermore if we enumerate $\L[2]$ we see that $\Delta \C(1) = 1$.
These two facts imply that for $n \geq 1$ we have $\Delta \C(n) = 1$.
As the language $\mathcal{L}(\vv)$ is closed under reversal, the language $\L$ is closed under $\Theta$.

To prove the rest we can either apply directly \Cref{criterion_t-epis} or enumerate $\Theta$-palindromic complexity.
\end{priklad}

In the following examples the mapping $\Theta$ is considered the same as in \Cref{slavny_pripad}.

\begin{priklad}
Let $\pi$ be the following morphism
$$
\pi: 0 \mapsto a'cacc, 1 \mapsto a'cac.
$$

We set $\u = \pi(\vv)$ where $\vv$ is a Sturmian word over $\{0,1\}$.
Then $\u$ is $\Theta$-episturmian and not $\Theta$-rich.

The proof is similar as in \Cref{slavny_pripad}.
From $\vv$ we can deduce that the word $\u$ is aperiodic, uniformly recurrent and contains infinitely many $\Theta$-palindromes.
The last two properties imply that $\L$ is closed under $\Theta$.
As $\Delta \FC{2} = 1$ and there is one special factor of length $1$ -- the factor $c$ -- it is clear that $\u$ is $\Theta$-episturmian.

$\u$ is not $\Theta$-rich: factors $ca'c$ and $cac$ do not have a unioccurrent $\Theta$-palindromic suffix.

\end{priklad}

Next example is a $\Theta$-rich word which is not $\Theta$-episturmian.

\begin{priklad}
Set $\pi$ as follows
$$
\pi: 0 \mapsto aa', 1 \mapsto acca'.
$$
Let $\vv$ be a Sturmian word over $\{0,1\}$.
Then the word $\pi(\vv)$ is $\Theta$-rich and is not $\Theta$-episturmian.
The proof is left to the reader.

\end{priklad}

The following two examples concern periodic infinite words.
\begin{priklad}
The word $\left( caca' \right)^{\omega}$ is periodic but not $\Theta$-rich.
The factors $ca'c$ and $cac$ do not have a unioccurrent $\Theta$-palindromic suffix as before.
\end{priklad}

\begin{priklad}
The word $\left( ccaa' \right)^{\omega}$ is periodic and $\Theta$-rich.
It follows from the fact that we can find a unioccurrent $\Theta$-palindromic suffix for each prefix of this word.
\end{priklad}

\section{Final remarks}

A wider class of infinite words related to $\Theta$-palindromes has been studied in \cite{BuLuLuZa2}.
The class is constructed using the \textit{right $\Theta$-palindromic closure operator.}
To any word $w$ in $\A^*$ it associates the word $w^{\oplus}$ which is the shortest $\Theta$-palindromic word having $w$ as prefix.
It is clear that if $s$ denotes the longest $\Theta$-palindromic suffix of $w$ then if we denote $w=ps$ we have $w^{\oplus} = ps\Theta(p)$.

Let $u_0$ be a fixed finite word called \textit{seed}.
Let $\Psi: \A^* \to \A^*$ be a map defined as follows
\begin{eqnarray*}
\Psi(\varepsilon) & = & u_0, \\
\Psi(ux) & = & \left ( \Psi(u)x \right )^{\oplus} \,\, \text{for $u \in \A^*$ and $x \in \A$.}
\end{eqnarray*}

For a given infinite word $t$ we can define $\Psi(t)$ as
\[
\Psi(t) = \lim_{n \to +\infty} \Psi \left({\rm pref}_n (t) \right)
\]
where ${\rm pref}_n$ stands for the prefix of length $n$.
The word $t$ is called the \textit{directive word} of $\Psi(t)$ and is denoted $\Delta(\Psi(t))$.
The word $\Psi(t)$ is called \textit{$\Theta$-standard word with seed} if $u_0 \neq \varepsilon$ otherwise
it is called just $\Theta$-standard.
%We say that a $\Theta$-standard word is \textit{with seed} if $u_0 \neq \varepsilon$.
$\Theta$-episturmian words are included in $\Theta$-standard words with seed (exact condition and details are given in \cite{BuLuLuZa2}).

As the definition of $\Theta$-standard word is constructive,
a rough and direct condition to decide whether a $\Theta$-standard word is $\Theta$-rich is easy to derive.
If the word $\Psi \left({\rm pref}_n (t) \right)t_{n}$ is $\Theta$-rich for all $n \geq 0$ (with $t = t_0t_1t_2\ldots$),
then $\Psi(t)$ is $\Theta$-rich.

The following proposition connects the class of $\Theta$-standard words with seed with the class of $\Theta$-rich words.

\begin{prop}[Proposition 4.7, \cite{BuLuLuZa2}]
Let $\u$ be a $\Theta$-standard word with seed.
Then there exists an integer $M$ such that any prefix $p$ of $\u$ of length greater than $M$ has a unioccurrent $\Theta$-palindromic suffix.
\end{prop}

The integer $M$ is explained in detail in the paper.
It is clear that if $M = 1$, then $\u$ is $\Theta$-rich.
However $M = 1$ implies some restrictions on $\Theta$ and a criterion whether a $\Theta$-standard word with seed is $\Theta$-rich
remains an open question.

\section{Acknowledgements}
I would like to thank Edita Pelantová and L\!'ubomíra Balková for their careful reviewing and fruitful suggestions.
I acknowledge financial support by the Czech Science Foundation grant no. 201/09/0584,
by the grant no. MSM6840770039 and LC06002 of the Ministry of Education, Youth, and Sports of the Czech Republic and
by the grant no. SGS10/085OHK4/1T/14 of the Grant Agency of the Czech Technical University in Prague.

\bibliographystyle{amsplain}
\bibliography{pseudopal}

\end{document}